\documentclass{amsart}
\usepackage{amssymb}
\usepackage{graphicx}
\usepackage{amsmath}

\newtheorem{theorem}{Theorem}[section]
\newtheorem{lemma}[theorem]{Lemma}
\theoremstyle{definition}
\newtheorem{definition}[theorem]{Definition}
\newtheorem{proposition}[theorem]{Proposition}

\theoremstyle{remark}
\newtheorem{remark}[theorem]{Remark}
\numberwithin{equation}{section}

\newcommand{\Rn}{\mathbb{R}^{n}}
\newcommand{\Cn}{\mathbb{C}^{n}}

\newcommand{\blankbox}[2]{  \parbox{\columnwidth}{\centering
    \setlength{\fboxsep}{0pt}  }}

\DeclareMathOperator*{\limind}{lim\,ind}
\begin{document}
\title{On Polya's Theorem in Several Complex Variables}
\author{OZAN G\"UNY\"UZ}
\address{Sabanc{\i} University, 34956
Tuzla/Istanbul, Turkey}
\email{ozangunyuz@sabanciuniv.edu}
\author{VYACHESLAV ZAKHARYUTA}
\address{Sabanc{\i} University, 34956
Tuzla/Istanbul, Turkey}
\email{zaha@sabanciuniv.edu}

\begin{abstract}
 Let $K$ be a compact set in $\mathbb{C}$, $f$ a function analytic
in $\overline{\mathbb{C}}\smallsetminus K$ vanishing at $\infty $. Let $%
f\left( z\right) =\sum_{k=0}^{\infty }a_{k}\ z^{-k-1}$ be its Taylor
expansion at $\infty $, and $H_{s}\left( f\right) =\det \left(
a_{k+l}\right) _{k,l=0}^{s}$ the sequence of Hankel determinants. The classical
Polya inequality says that
\[
\limsup\limits_{s\rightarrow \infty }\left\vert H_{s}\left( f\right)
\right\vert ^{1/s^{2}}\leq d\left( K\right) ,
\]%
where $d\left( K\right) $ is the transfinite diameter of $K$. Goluzin has
shown that for some class of compacta this inequality is sharp. We provide
here a sharpness result for the multivariate analog of Polya's
inequality, considered by the second author in Math. USSR Sbornik,
25 (1975), 350-364.
\end{abstract}

\maketitle







\section{Preliminaries and Introduction }

We denote by \thinspace $A(\mathbb{C}^{n})^{*}$ the dual space to the
space $A(\mathbb{C}^{n})$ of all entire functions on \thinspace $\mathbb{C}%
^{n}$, \thinspace\ equipped with the locally convex topology of locally
uniform convergence in\thinspace $\mathbb{C}^{n}$. Following H\"{o}%
rmander(\cite{H}, Section 4.5), we call the elements of \thinspace $A(%
\mathbb{C}^{n})^{*}$\thinspace\ \textit{analytic functionals}.

Let \thinspace $\mathbb{Z}_{+\text{ }}^{n}$\thinspace\ be the collection of
all $n$-dimensional vectors with non-negative integer coordinates. For
\thinspace $k=\left( k_{1},\ldots ,k_{\nu },\ldots ,k_{n}\right) \in \mathbb{%
Z}_{+\text{ }}^{n}$ and $z=(z_{1},\ldots ,z_{n})\in \mathbb{C}^{n}$, let $%
z^{k}=z_{1}^{k_{1}}\ldots z_{n}^{k_{n}}$ and ${\lvert k\rvert }%
:=k_{1}+\ldots +k_{n}$ be the degree of the monomial $z^{k}$. We consider $\
$ the enumeration $\left\{ k\left( i\right) \right\} _{i\in \mathbb{N}}$ of
the set $\mathbb{Z}_{+\text{ }}^{n}$ such that ${\lvert k(i)\rvert }\leq {%
\lvert k(i+1)\rvert }$ and on each set $\left\{ \left\vert k\left( i\right)
\right\vert =s\right\} $ the enumeration coincides with the lexicographic
order relative to $k_{1},\ldots ,k_{n}$. We will write $s(i):={\lvert
k(i)\rvert }$. The number of multiindices of degree at most $s$ is
\thinspace $m_{s}:=C_{s+n}^{s}$\thinspace and the number of those of degree
exactly $s$ is $N_{s}=m_{s}-m_{s-1}=C_{n+s-1}^{s}$,\thinspace \thinspace $%
s\geq 1$;\thinspace \thinspace $N_{0}=1$. Let \thinspace $%
l_{s}:=\sum_{q=0}^{s}qN_{q}$\thinspace \thinspace\ for $s=0,1,2,\ldots $.

Consider Vandermondians:
\begin{equation*}
V\left( \zeta _{1},\ldots ,\zeta _{i}\right) :=\det \left( e_{\alpha }\left(
\zeta _{\beta }\right) \right) _{\alpha ,\beta =1}^{i},\ i\in \mathbb{N}%
\text{,}
\end{equation*}%
where $e_{\alpha }\left( z\right) :=z^{k\left( \alpha \right) },\ \alpha \in
\mathbb{N}$ and $\left( \zeta _{\beta }\right) \in K^{i}$.

For a compact set $K\subset \mathbb{C}^{n}$, define "maximal Vandermondians":
\begin{equation*}
V_{i}:=\sup \left\{ \left\vert V\left( \zeta _{1},\ldots ,\zeta _{i}\right)
\right\vert :\left( \zeta _{j}\right) \in K^{i} \right\},\,\,\,i\in \mathbb{N} .
\end{equation*}%
Set $d_{s}\left( K\right) :=\left( V_{m_{s}}\right) ^{1/l_{s}}$. The
transfinite diameter of $K$ is the number:
\begin{equation}
d\left( K\right) :=\limsup\limits_{s\rightarrow \infty }d_{s}\left( K\right)
.  \label{dk}
\end{equation}%
In the one-dimensional case, this notion was introduced by Fekete \cite{F}
for $n=1$, and by Leja \cite{L} for $n\geq 2$. That, in fact, the usual
limit can be taken in (\ref{dk}) was proved in \cite{F} for $\ n=1$ and in
\cite{Za1} for $n\geq 2$.

The \textit{pluripotential\ Green\ function}$\emph{\ }$of a compact set $%
K\subset \mathbb{C}^{n}$ is defined as follows
\begin{equation*}
g_{K}(z)=\limsup\limits_{\zeta \rightarrow z}\sup \{u(\zeta):u|_{K}\leq 0,\ u\in
\mathcal{L}(\Cn)\},
\end{equation*}%
where \thinspace $\mathcal{L}(\mathbb{C}^{n})$ represents the Lelong class consisting of
all functions\thinspace\ $u\in {Psh}(\mathbb{C}^{n})$ such that $u(\zeta
)-\ln |\zeta |$ is bounded from above near infinity. Since $d(K)=d(\hat{K}),$ with no loss of generality, we are going to consider polynomially convex compact sets. We will also consider the class of functions
$\mathcal{L}^{+}(\mathbb{C}^{n}):=\{u\in \mathcal{L}(\mathbb{C}^{n}): u(z)\geq log^{+}\lvert z \rvert
+C\}$. The function $g_{K}(z)$
is either plurisubharmonic in $\mathbb{C}^{n}$ or identically equal to $%
+\infty $. For more detail about the
pluripotential Green function, we refer the reader to \cite{Kl}, \cite{Sa} and \cite{Za2}.

The \emph{Monge-Ampere energy $\mathcal{E}(u, v)$ of u relative to v} for $u, v\in
\mathcal{L}^{+}(\mathbb{C}^{n})$ is defined as follows (\cite{BB}, Section 5): \begin{equation*}
\mathcal{E}(u,v):= \int_{\mathbb{C}^{n}}{(u-v)\sum_{j=0}^{n}{(dd^{c}u)^j\wedge(dd^{c}v)^{n-j}}}.
\end{equation*}

   Let $K$ be a compact set in $\mathbb{C}^{n}$. $A(K)$ represents the
locally convex space of all germs of analytic functions on $K$, equipped
with the countable inductive limit topology, i.e., $$A(K)=\limind_{j\rightarrow \infty}{A(D_{j})}$$
\thinspace considered in regard to the inclusion of sets. \thinspace $D_{j}$\thinspace\ are open sets
such that $D_{j+1}\Subset D_{j}$ for each $j\in \mathbb{N}$ \thinspace\ and \thinspace $%
K=\bigcap_{j=1}^{\infty }{D_{j}}$. Thus, in this setting, a sequence $\{u_{j}\}$ of germs
converges to a germ $u$ in this topology in case there exists an open
neighbourhood $V\supset K$ and functions $g_{j},\,\,g\in A(V)$ being the
representatives of the germs $u_{j},\,\,u$ respectively, such that $g_{j}$
converges uniformly to $g$ on any compact subset of $V$.

 The Polya Theorem (Theorem \ref{polya1}) and its multivariate analog (Theorem \ref{polya2}), considered by the second
author in \cite{Za1}, are discussed in Section 2. The sharpness result of the
generalized Polya inequality (section 4) is based on the comparison of the
expression (\ref{hank}) for Hankel-like determinants from \cite{Za1} with the expression
(\ref{zskm}) for the transfinite diameter from Bloom and Levenberg \cite{BL}. The main
result of this article (Theorem \ref{strongerr}) says that, for \textit{real} compact sets,
the equality \textit{attains} in the generalized Polya inequality (\ref{polyainn}) for
some analytic functional $f^{*}\in A(K)^{*}$. This
result seems to be new even in the one-dimensional case. Additionally, we introduce two sharpness properties for compact sets $K\subseteq \Cn$ and study the stability of these properties relative to the approximations from inside and outside (Proposition \ref{wsp} and \ref{sharpstablein}).

\section{Polya's Theorem}

The following result is due to Polya \cite{Po}.

\begin{theorem}\label{polya1}
Let $K$ be a polynomially convex compact set in $\mathbb{C}$ and $f$ $\in
A\left( \overline{\mathbb{C}}\smallsetminus {K}\right) $ have the following
expansion in a neighbourhood of $\infty $:
\begin{equation}
f(z)=\sum_{k=0}^{\infty }{\frac{a_{k}}{z^{k+1}}.}  \label{texp}
\end{equation}%
Let $A_{s}\left( f\right) :=\det (a_{k+m})_{k,m=0}^{s-1}$, $s\in \mathbb{N}$
, be a sequence of Hankel determinants composed from the coefficients of the
expansion (\ref{texp}). Then,
\begin{equation}
D\left( f\right) :=\limsup_{s\rightarrow \infty }{\lvert A_{s}}\left(
f\right) {\rvert }^{1/s^{2}}\leq d(K).  \label{polineq}
\end{equation}
\end{theorem}

\vspace{1mm}

A direct multivariate analog of the inequality (\ref{polineq}) makes no
sense, since there are functions analytic on the complement of $K$ but
constants only. Schiffer and Siciak\thinspace \thinspace (\cite{SS}) proved some
analog for the product of plane compact sets \thinspace\ $K=K_{1}\times
K_{2}\times \ldots \times K_{n}\subset \mathbb{C}^{n}$ and functions $f\in
A((\overline{\mathbb{C}}-K_{1})\times \ldots \times (\overline{\mathbb{C}}%
-K_{n}))$. Sheinov (\cite{Sh2}, \cite{Sh}) considered another analog of Polya's inequality
for a \textit{linearly convex} compact set $K$, considering the Taylor
expansion at the origin for functions analytic in the domain $D=K^{\ast }$
linearly convex adjoint (conjugate) to $K$ (projective complement of $K$ by
Martineau \cite{Ma}).

The case of an arbitrary compact set $K\subset \mathbb{C}^{n}$ was studied
in \cite{Za1}. It was suggested there, instead of analytic functions on some
artificial "complement" of $K$, to deal with those analytic functionals in $%
\mathbb{C}^{n}$ that are extendible continuously
onto the space $A(\widehat{K})$. We denote by $A_{0}(\{\infty^{n}\})$ the space of all analytic germs
$f'$ at $\infty^{n}=(\infty, \infty, \ldots, \infty)\in \overline{\mathbb{C}^n}$ \,with Taylor expansion
of the form \begin{equation} \label{expan}f'(z)=\sum_{k\in \mathbb{Z}^n}{\frac{a_{k}}{z^{k+I}}},\,\,\,\,I=(1, 1,...,
1),\end{equation}
converging in some neighborhood of $\infty ^{n}.$

\begin{lemma}\label{isomm} There is an isomorphism,
\begin{equation}
T:A\left( \mathbb{C}^{n}\right) ^{\ast }\rightarrow A_{0}\left( \left\{ \infty
^{n}\right\} \right),  \label{tiso}
\end{equation}
such that, for each $f^*$ and $f'=Tf^*$, we have
\begin{equation*}
f^{\ast }\left( \varphi\right) =\left\langle \varphi,f^{\prime }\right\rangle :=\left(
\frac{1}{2\pi i}\right) ^{n}\int_{\mathbb{T}_{R}^{n}}\varphi\left( \zeta \right)
f^{\prime }\left( \zeta \right) \ d\zeta ,\ \varphi\in A\left( \mathbb{C}%
^{n}\right),
\end{equation*}
where
\begin{equation}\label{torusn}
\mathbb{T}_{R}^{n}:=\left\{ z=\left( z_{\nu }\right) \in \mathbb{C}
^{n}:\left\vert z_{\nu }\right\vert =R,\ \nu =1,\ldots ,n\right\} ,\
R=R\left( f^{\ast }\right) \text{.}
\end{equation}
\end{lemma}

\begin{proof} See, e.g., \cite{Sd}, Chapter 3. \end{proof}

Let us define, for every analytic functional $f^{\ast }$, a related sequence
of multivariate Hankel-like determinants constructed from the coefficients
of the expansion (\ref{expan}):
\begin{equation}
H_{i}=H_{i}\left( f^{\ast }\right) :=\det \left( a_{k\left( \alpha \right)
+k(\beta) }\right) _{\alpha ,\beta =1}^{i},\ i\in \mathbb{N}  \label{hank}
\end{equation}%
with $a_{k\left( \alpha \right) }:=f^{\ast }\left( e_{\alpha }\right)
=\left\langle e_{\alpha },f^{\prime }\right\rangle ,\ \alpha \in \mathbb{N},\,\,f'=Tf^*$%
. Now we are ready to formulate the general form of multivariate Polya's
inequality.

\begin{theorem}\label{polya2}
Suppose \thinspace that $K$\thinspace\ is a compact set in \thinspace $%
\mathbb{C}^{n}$, \thinspace $f^{\ast }$ is an analytic functional which has
a continuous extension onto $A\left( K\right)$ and $f'=Tf^{*}$ is the corresponding analytic germ at
$\infty^{n}$. Then for the determinants (%
\ref{hank}), the inequality holds:
\begin{equation}\label{polyainn}
D(f'):= \limsup_{i\rightarrow \infty }{{\lvert H_{i}}}\left( f^{\ast }\right) {{%
\rvert }^{\frac{1}{2l_{s(i)}}}}\leq d(K).
\end{equation}
\end{theorem}

It has been proved in \cite{Za1} a bit weaker result with the outer
transfinite diameter $\widehat{d}\left( K\right) $ instead of $d(K)$, but
later it was proved that $\widehat{d}\left( K\right) =d(K)$\,(see Proposition \ref{stabilitydown} below).

We send the reader for the proof of Theorem \ref{polya2} to \cite{Za1}, Theorem 3. However we cite here the
following equality, which is crucial there and will be essentially used in Section 4:
\begin{equation}
i!\,{\lvert \ H\,}_{i}\left( f^{\ast }\right) {\rvert }={\lvert \ f_{\zeta
^{(i)}}^{\ast }(\ldots f_{\zeta ^{(j)}}^{\ast }(\ldots (f_{\zeta
^{(1)}}^{\ast }({[V(\zeta ^{(1)},\zeta ^{(2)},\ldots ,\zeta ^{(i)})]}%
^{2})\ldots )\ldots )\rvert ,}  \label{i!hi}
\end{equation}%
$i\in \mathbb{N}$, here the notation \thinspace $f_{\zeta ^{(j)}}^{\ast }$%
\thinspace\ means that the functional $f^{\ast }$ is applied sequentially to
a function of the variable $\zeta ^{(j)}$ by keeping the other variables
fixed.
\begin{remark}
The classical Polya's Theorem (Theorem \ref{polya1}) is a particular case of Theorem \ref{polya2} since, due to
Gr\"{o}thendieck-K\"{o}the-Silva duality\,(see \cite{Gr}, \cite{Kot}, \cite{Sil}), every $f\in
A(\overline{\mathbb{C}}\backslash K)$ satisfying (\ref{texp}) in a neighborhood of $\infty$ represents a linear
continuous functional $f^{*}\in A(K)^{*}\hookrightarrow A(\mathbb{C})^{*}$. Hereafter $\hookrightarrow$ denotes
a linear continuous embedding.
\end{remark}

Let \thinspace $K\subset \mathbb{C}^{n}$\thinspace\ be a  compact set, and \thinspace $\mu $\thinspace\ be a bounded positive Borel measure on \thinspace $K$. The pair \thinspace $(K,\mu )$ \thinspace\ is
said to satisfy \textit{Bernstein-Markov inequality} for holomorphic
polynomials in \thinspace $\mathbb{C}^{n}$\thinspace\ if, given \thinspace $%
\varepsilon >0$\thinspace ,\thinspace\ there exists a constant \thinspace $%
M=M(\varepsilon )$\thinspace\ such that for all polynomials \thinspace $%
p_{s} $\thinspace\ of degree at most \thinspace $s$\thinspace\
\begin{equation*}
{\lVert \ p_{s}\rVert }_{K}\leq M(1+\varepsilon )^{s}{\lVert \ p_{s}\rVert }%
_{L^{2}(\mu )}.
\end{equation*}

\begin{theorem}\label{BL25}
(Bloom-Levenberg, \cite{BL})\thinspace Let \thinspace $K\subset \mathbb{C}^{n}$%
\thinspace\ be a compact set,\thinspace \thinspace $\mu $%
\thinspace\ be a bounded positive Borel measure on \thinspace $K$\thinspace\ and
let  \thinspace $(K,\mu )$\thinspace\ satisfy Bernstein-Markov inequality.
Then,
\begin{equation*}
\lim_{s\rightarrow \infty }{Z_{s}(K,\mu )^{\frac{1}{2l_{s}(n)}}}=d(K),
\end{equation*}%
where
\begin{equation}
Z_{s}(K,\mu )=\int_{K^{m_{s}(n)}}{\lvert \ V(\zeta ^{(1)},\ldots ,\zeta ^{(m_{s}(n))})\rvert }^{2}d\mu (\zeta ^{(1)})d\mu\ldots d\mu (\zeta ^{(m_{s}(n))}).  \label{zskm}
\end{equation}
\end{theorem}
\begin{remark}\label{existmu}In \cite{BL2} (Proposition 3.4 and Corollary 3.5), the same authors proved that for any compact set $K\subseteq \Cn$, there exists a measure $\mu\in \mathcal{M}(K)$ such that $(K, \mu)$ satisfies  Bernstein-Markov property. \end{remark}

\section{Stability of Transfinite Diameter}
The following proposition provides the stability of  transfinite diameter of a compact set in
$\mathbb{C}^n$ approximated from outside.
\begin{proposition}(V.A. Znamenskii \cite{Zn1,Zn2}, Levenberg \cite{Le1})\label{stabilitydown}
 Let $K$ be a compact set in $\mathbb{C}^n$ and  $\{K_{j}\}$ a sequence of compact sets such that
 $K_{j+1}\subseteq K_{j}$ for all $j\in \mathbb{N}$  and  $K=\bigcap_{j=1}^{\infty}{K_{j}}$. Then,
 $$\widehat{d}(K):=\lim_{j\rightarrow \infty}{d(K_{j})}=d(K).$$
\end{proposition}

In this section, we prove a stability property of transfinite diameter relative to the approximation
from inside. The following is an easy consequence of Lemma 6.5 in \cite{BT}:

\begin{lemma}\label{greendown}
Suppose that $K$ is a non-pluripolar compact set in  $\mathbb{C}^n$, and $\{K_{j}\}$  is a sequence of
non-pluripolar compact sets such that $K_{j}\subseteq K_{j+1}\subseteq K$, $j\in \mathbb{N}$ and for
$L:=\bigcup_{j=1}^{\infty}{K_{j}}$, we have \begin{equation}\int_{K\backslash L}{(dd^{c}g_{K})^{n}}=0.
\end{equation} Then $$\lim_{j\rightarrow \infty}{g_{K_{j}}(z)}=g_{K}(z),\,\,\,\,z\in \mathbb{C}^{n}.$$
\end{lemma}

\begin{theorem}\label{stabilityup}
Under the conditions of Lemma \ref{greendown}, we have,$$\lim_{j\rightarrow \infty} {d(K_{j})}= d(K).$$
\end{theorem}

\begin{proof} We will use the unweighted energy version of Rumely's formula (see e.g., Theorem 5.1 of
\cite{Le2}, or Section 9.1 of \cite{BB}). Since, by Lemma \ref{greendown}, $g_{K_{j}}\downarrow g_{K}$, applying the remark after Lemma 3.5 in \cite{Le2}, one obtains
$$-\ln{d(K_{j})}=\frac{1}{n(2\pi)^{n}}\mathcal{E}(g_{K_{j}}, g_{T})\uparrow
\frac{1}{n(2\pi)^{n}}\mathcal{E}(g_{K}, g_{T})=-\ln{d(K)}, \,\,as\,\,\,j\rightarrow \infty,$$  where $T$
is the unit torus in $\mathbb{C}^{n}.$\end{proof}

\section{Sharpness of Polya's Inequality}
The following Theorem is proved by Goluzin in \cite{G1} (see also \cite{G2}, Section 11).
\begin{theorem}\label{sharp1}
For functions which are analytic in an infinite domain $B$ with boundary K consisting of a finite number of closed Jordan curves and having the expansion \begin{equation}f(z)=\sum_{k=1}^{\infty}{\frac{a_{k}}{z^{k}}},\end{equation} in a neighborhood of $z=\infty$, the inequality $D(f)=\limsup_{s\rightarrow \infty }{\lvert A_{s}}\left(
f\right) {\rvert }^{1/s^{2}}\leq d(K)$ given by Theorem \ref{polya1} is sharp.
\end{theorem}

Another way of expressing Theorem \ref{sharp1} is, for a compact set $K\subseteq \mathbb{C}$
\begin{equation}
d\left( K\right) =\sup \left\{ D\left( f\right) :f\in A\left( \overline{%
\mathbb{C}}\smallsetminus K\right) \right\} ,
\end{equation}%
if the boundary $\partial K$ consists of a finite number of closed
Jordan curves.

\begin{definition}\label{newdesc} Let $K$ be a polynomially convex compact set in $\Cn$. $K$ is said to satisfy the \emph{sharpness property} in Polya inequality, shortly denoted as $K\in (SP)$, if $$d(K)=\sup\{D(f'): f'=T(f^{*}), f^{*}\in A(K)^{*}\}.$$
We say that $K$ has a \emph{strong sharpness property} in Polya inequality, denoted by $K\in (SSP)$,  if there exists a $f^{*}\in A(K)^{*}$ such that $$D(f')=d(K)$$ for $f'=T(f^{*})$, where $T$ is defined as in Lemma \ref{isomm}.\end{definition}

If $K$ is a pluripolar compact set in $\Cn$, then $K\in (SSP)$ by the result of Levenberg-Taylor (\cite{LT}) which says that  $d(K)=0$ if and only if $K$ is pluripolar. From now on, we only consider non-pluripolar compact sets.

\begin{proposition}\label{wsp} Let $K$ be a compact set in $\Cn$, $\{K_{i}\}$ a sequence of compact sets with $K=\bigcap_{i=1}^{\infty}{K_{i}}$. Assume $K_{i}\in (SP)$ for all $i\in \mathbb{N}.$ Then there exists a sequence of analytic functionals $\{f^{*}_{i}\}$ such that $f^{*}_{i}\in A(K_{i})^{*}$ for each $i\in \mathbb{N}$ and \begin{equation}\label{nwsp}\lim_{i\rightarrow \infty}{D(f'_{i})}=d(K).\end{equation} \end{proposition}

\begin{proof} By Definition \ref{newdesc}, for each $i\in \mathbb{N}$, there exists $f^{*}_{i}\in A(K_{i})^{*}$ with $f'_{i}=T(f_{i}^{*})$ such that $d(K_{i})\leq D(f'_{i})+\frac{1}{i}$. Theorem \ref{polya2} gives $D(f'_{i})\leq d(K_{i})$. By using Proposition \ref{stabilitydown}, we have $$d(K)=\lim_{i\rightarrow \infty}{d(K_{i})}\leq \lim_{i\rightarrow \infty}{D(f'_{i})}\leq \lim_{i\rightarrow \infty}{d(K_{i})}=d(K),$$ which gives the limit (\ref{nwsp}). \end{proof}

As seen from Proposition \ref{wsp}, $(SP)$ is not preserved under the approximation from outside, however, for an approximation from inside, we have the stability of the property $(SP):$

\begin{proposition}\label{sharpstablein} Let the conditions of Lemma \ref{greendown} be given. Suppose further that $K_{i}\in (SP)$ for all $i\in \mathbb{N}.$ Then $K\in (SP).$\end{proposition}

\begin{proof} Proof is almost the same as the proof of Proposition \ref{wsp} except we only use Theorem \ref{stabilityup} instead of Proposition \ref{stabilitydown} in the end, hence we have the following: $$d(K)\leq \lim_{i\rightarrow \infty}{D(f'_{i})}=\sup\{D(f'_{i}): i\in \mathbb{N}\}\leq d(K),$$ which concludes that $d(K)=\sup\{D(f'_{i}): i\in \mathbb{N}\}$ and so $K\in (SP)$ by Definition \ref{nwsp}.  \end{proof}

For an arbitrary compact set in $\mathbb{C}$, a following sharpness statement, which is weaker than $(SP)$,
is derived easily from Goluzin's result above.

\begin{proposition}\label{weaker1}
Let $K$ be a compact set in $\mathbb{C}$, $\{K_{i}\}$ a sequence of compact sets with the properties $K_{i+1}\Subset K_{i}$  for all $i\in \mathbb{N}$, $K=\bigcap_{i=1}^{\infty}{K_{i}}$. Then there exists a sequence of  functions $f_{i}\in A(\overline{\mathbb{C}}\backslash K_{i})$  such that
\begin{equation}\label{gollimit}\lim_{i\rightarrow \infty}{D(f_{i})}=d(K).\end{equation}
\end{proposition}

\begin{proof} For each $i\in \mathbb{N}$, we can find a compact set $L_{i}$ whose boundary consists of a finite number of closed analytic Jordan curves so that $K_{i+1}\Subset L_{i}\Subset K_{i}$ holds. By the result of Goluzin, there exists  $f_{i}\in A(\overline{\mathbb{C}}\setminus L_{i})$ such that, $d(L_{i})< D(f_{i})+\frac{1}{i}$. Since $f_{i}\in A(\overline{\mathbb{C}}\setminus K_{i})$ holds, we get by Theorem \ref{polya1}, $D(f_{i})\leq d(K_{i})$. Hence by using Proposition \ref{stabilitydown} we obtain the following $$d(K)=\lim_{i\rightarrow \infty}{d(L_{i})}\leq \lim_{i\rightarrow \infty}{D(f_{i})}\leq \lim_{i\rightarrow \infty}{d(K_{i})}=d(K),$$ which gives the desired limit (\ref{gollimit}).\end{proof}

Let $K$ be a compact set in $\Cn$, and $J:A(K)\rightarrow C(K)$ the natural restriction homomorphism.
$AC(K)$ is the Banach space obtained as the completion of the set $J(A(K))$ in the space $C(K)$ with
respect to the uniform norm.

\begin{lemma}\label{functional}
Let $K$  be an infinite polynomially convex compact set in $\mathbb{C}^{n}$.  Then, for each
bounded Borel measure $\mu\in \mathcal{M}(K)$,  there exists an analytic functional $f^{*}\in
A(K)^{*}\hookrightarrow A(\mathbb{C}^n)^{*}$ and a corresponding analytic germ $f'=Tf^*$  such that
\begin{equation}\label{funcmu}f^{*}(f)=<f,
f'>=\int_{K}{f(\zeta)}d\mu(\zeta), \end{equation} for every $f\in A(\Cn)$.
\end{lemma}

\begin{proof}
The dense embedding \,$A(K)\hookrightarrow AC(K)$ \,implies, for the dual spaces, the following
embedding:
$AC(K)^{*}\hookrightarrow A(K)^{*}$. Since $AC(K)$  is a closed subspace of  $C(K)$,  every bounded
Borel
measure  $\mu\in \mathcal{M}(K)$  defines a linear continuous functional  $F^{*}\in AC(K)^{*}$  such
that $$F^{*}(f)=\int_{K}{f(\zeta)d\mu(\zeta)}$$ for every $f\in AC(K)$. Then, the restriction
$f^{*}=F^{*}|_{A(K)}$  belongs to  $A(K)^{*}$. By Lemma \ref{isomm}, since  $A(K)^{*}\hookrightarrow A(\Cn)^{*}$,  there is $f'\in A_{0}(\{\infty^{n}\})$  such that $$f^{*}(f)=<f, f'>=\left(
\frac{1}{2\pi i}\right) ^{n}\int_{\mathbb{T}_{R}^{n}}f\left( \zeta \right)
f^{\prime }\left( \zeta \right) \ d\zeta ,\ f\in A\left( \mathbb{C}%
^{n}\right),$$ where $\mathbb{T}_{R}^{n}$ \, is defined as in (\ref{torusn}), and $R$ is sufficiently large.
\end{proof}

Now we show that, for any real compact set in $\Cn$, the equality in the estimate (\ref{polyainn}) is
attained at some $f^{*}\in A(K)^{*}$.
\begin{theorem}\label{strongerr}
Let $K\subseteq \Rn\subseteq \Cn$\, be a compact set. Then $K\in (SSP)$.
\end{theorem}
\begin{proof}
By Theorem \ref{BL25} and Remark \ref{existmu}, there exists a measure $\mu\in \mathcal{M}(K)$ such that   $(K,\mu)$  satisfies the Bernstein-Markov
inequality. Let $f^{*}$ be an analytic functional corresponding to $\mu$ by Lemma \ref{functional}. Initially, we show that \,$Z_{s}(K, \mu)=m_{s}(n)!\,{\lvert\
H_{m_{s}(n)}{(f^{*})} \rvert}$, \,where \, $Z_{s}(K, \mu)$\, and\, $%
H_{m_{s}(n)}{(f^{*})}$ are defined in Section 2. Indeed, considering the relation (\ref{i!hi}) gives :
\begin{equation}\label{inthank}
m_{s}(n)!\,{\lvert H\,_{m_{s}(n)}{(f^{* })}\rvert} = {\lvert\
f^{*}_{\zeta^{(m_{s}(n))}}(\ldots(f^{*}_{\zeta^{(1)}}({
[V(\zeta^{(1)},\ldots, \zeta^{(m_{s}(n))}})]^{2})\ldots)
\rvert},
\end{equation}
Since $K$ is a real subset and so ${[V(\zeta^{(1)},\zeta^{(2)},\ldots, \zeta^{(m_{s}(n))}})]^{2}$ is nonnegative, by iterating (\ref{funcmu}) \,$m_{s}(n)$\, times, the righthand side of (\ref{inthank}) becomes :
\begin{equation*}
\int_{K}\ldots \int_{K}{\lvert\ V(\zeta^{(1)},\zeta^{(2)},\ldots,
\zeta^{(m_{s}(n))}) \rvert}^{2} d\mu(\zeta^{(1)})\ldots
d\mu(\zeta^{(m_{s}(n)}),
\end{equation*}
which is equal to $Z_{s}(K, \mu)$. Since \,$(m_{s}(n)!)^{\frac{1}{2l_{s}(n)}%
}\rightarrow 1$\,\, as \,$s\rightarrow \infty$,\, we have, by Theorem \ref{BL25},
\begin{equation*}
d(K)=\lim_{s\rightarrow \infty}{Z_{s}(K, \mu)}^{\frac{1}{2l_{s}(n)}%
}=\lim_{s\rightarrow \infty}{{\lvert\ H\,_{m_{s}(n)}{(f^{*})} \rvert}}%
^{\frac{1}{2l_{s}(n)}}=D(f^{'}).
\end{equation*}
\end{proof}

\emph{$\mathbf{Problem.}$} Characterize compact sets in $\Cn$ such that either $(SP)$ or $(SSP)$ holds.

\subsection*{Acknowledgements}

We are grateful to the referee for attracting our attention to the Bloom-Levenberg article \cite{BL2} and for other valuable remarks.

%
%
%
%

\end{document}